\documentclass[12pt,amstex,leqno]{amsart}

\usepackage{latexsym}
\usepackage{amsthm}
\usepackage{amsmath}

\usepackage{amsfonts}
\usepackage{amssymb}
\usepackage{appendix}
\usepackage[mathcal]{euscript}

\newcommand\op{\operatorname{op}}

\newcommand\Alg{\operatorname{\bf Alg}}
\newcommand\Set{\operatorname{\bf Set}}
\newcommand\Pos{\operatorname{\bf Pos}}

\newcommand{\id}{\operatorname{id}}

\newcommand\obj{\operatorname{obj}}
\newcommand\coeq{\operatorname{coeq}}

\newcommand\ck{\mathcal {K}}
\newcommand\ca{\mathcal {A}}

\newcommand\cl{\mathcal {L}}

\newcommand\cv{\mathcal {V}}

\newcommand\N{\mathbb{N}}

\newcommand\R{\mathbb{R}}

\input xy
\xyoption{all}

\swapnumbers

\newtheorem{theorem}{Theorem}[section]
\newtheorem{lemma}[theorem]{Lemma}
\newtheorem*{lem}{Lemma}
\newtheorem{birk}[theorem]{Birkhoff Variety Theorem}

\newtheorem{prop}[theorem]{Proposition}
\newtheorem{corollary}[theorem]{Corollary}

\theoremstyle{definition}
\newtheorem{defi}[theorem]{Definition}

\newtheorem{example}[theorem]{Example}

\newtheorem{remark}[theorem]{Remark}

\newtheorem{constr}[theorem]{Construction}

\newtheorem{nota}[theorem]{Notation}

\numberwithin{equation}{section}

\begin{document}

	
	\title[Categories which are varieties]{Categories which are varieties of classical or ordered algebras}
	\author[Ji{\v{r}}{\'{\i}} Ad{\'{a}}mek]{Ji{\v{r}}{\'{\i}} Ad{\'{a}}mek$^{^{\ast)}}$}
	\dedicatory{\rm Czech Technical University, Prag\\
		Technical University of Braunschweig}

	\dedicatory{Dedicated to the memory  of Bill Lawvere}
	
	\begin{abstract}
		Following ideas of Lawvere and Linton we prove that  classical varieties are precisely the exact categories with a varietal generator. This means a strong generator which is abstractly finite and regularly projective.

		An analogous characterization of varieties of ordered algebras is also presented. We work  with order-enriched categories, and introduce the concept of subexact category and subregular projective  (corresponding naturally to the ordinary case). 
		Varieties of ordered algebras are precisely the subexact categories with a subvarietal generator. This means a strong generator which is abstractly finite and subregularly projective.
	\end{abstract}
	\footnote{$^{^\ast}$ Supported by the Grant Agency of Czech Republic, Grant No. 22-02964S.}
	
	\maketitle
	
	\section{Introduction}\label{sec1}
	One of the fundamental achievements of the thesis of Bill Lawvere was a characterization of categories equivalent to varieties of (finitary, one-sorted) algebras.
	He introduced the concept of an abstractly finite object $G$ (weaker than the concept of a finitely generated object, later used by Gabriel and Ulmer): every morphism from $G$ to its copower factorizes through a finite subcopower. 
	Lawvere formulated a theorem stating that a category is equivalent to a variety iff it has
	\begin{enumerate}
		\item[(1)] Finite limits.
		
		\item[(2)] Effective congruences.
		
		\item[(3)] A generator  with copowers which is  abstractly finite and regularly projective (its hom-functor preserves regular epimorphisms).
	\end{enumerate}
	Unfortunately, a small correction is needed: in (1) the existence of coequalizers should  be added (since Lawvere uses them twice in his proof), and the generator in (3) needs to be regular (also used in that proof). A category satisfying the three conditions above which, however, is not equivalent to a variety is presented below (Example \ref{E:cont}).
	
	Thus Lawvere's, very elegant, proof is a verification of  the following theorem.
	
	\begin{theorem}\label{th1.1}
		A category is equivalent to a variety iff it has 
		\begin{enumerate}
			\item[(1)] Finite limits and coequalizers.
			
			\item[(2)] Effective congruences.
			
			\item[(3)] A regular generator $G$  with copowers which is  an abstractly finite regular projective.
		\end{enumerate}
	\end{theorem}
	Actually, Lawvere worked in (2) with congruences effective with respect to $G$, but we prove in Proposition \ref{L:G} that this makes no difference in case $G$ is a regular generator. 
	An improved version of the above theorem was presented  in \cite{ARV}: kernel pairs and reflexive  coequalizers are sufficient in (1), and strong (rather than regular) generator in (3). This leads us to the following
	
	\begin{defi}\label{D:vg}
		A \emph{varietal generator} is a strong generator with copowers which is an abstractly finite regular projective.
	\end{defi} 
	
	By applying Linton's characterization of monadicity over $\Set$, we
	present a shorter proof and 
	make one further simplification step (Theorem \ref{T:main} below): in (1) coequalizers of kernel pairs are sufficient. This corresponds well to Barr's concept of an exact category (Def. \ref{D:Barr} below):
	he only assumed that kernel pairs and their coequalizers exit. We obtain the following result (Theorem \ref{T:main} below).

	\begin{theorem}\label{th1.3}
		A category is equivalent to a variety iff it is exact and has a varietal generator.
	\end{theorem}
	
	Our second topic is a characterization of varieties of ordered algebras. Here one works with algebras acting on posets so that the operations are monotone. A variety is a  full subcategory  presented by inequations between terms. Varieties are enriched categories over the cartesian closed category $\Pos$ of posets. In \cite{ARO} a characterization of varieties of ordered algebra has been presented, and our purpose is to sharpen and correct that result slightly . Whereas reflexive coequalizers play an important role in classical varieties (because they are preserved by the forgetful functor to $\Set$), for ordered varieties reflexive coinserters  (Def. \ref{D:coins}) play the analogous role.
	A \emph{subregular epimorphism} is a morphism which is a coinserter of a reflexive pair. An object whose hom-functor to $\Pos$ preserves subregular epimorphisms is a \emph{subregular projective}.
	
	\begin{defi}\label{de1.4}
		A \emph{subvarietal generator} in an order-enriched category is a strong generator with  copowers which is an abstractly finite subregular projective.
	\end{defi}
	
	Whereas in ordinary categories a congruence is a reflexive, symmetric and transitive relation, in order-enriched categories, we 
	lose the symmetry, but gain a stronger property then reflexivity -- we call it hyper-reflexivity (Def. \ref{D:hyper}). We
	introduce \emph{subcongruences}: relations that are hyper-reflexive and transitive.
	Example: given a morphism $f\colon X \to Y$, its subkernel pair $r_0$, $r_1 \colon R\to X$ (universal with respect to $f\cdot r_0 \leq f\cdot r_1$) is a subcongruence. We prove that every variety of ordered algebras has \emph{effective subcongruences}: each subcongruence is the subkernel pair of a morphism.
	
	\begin{defi}\label{de1.5}
		An order-enriched category is \emph{subexact} if it has subkernel pairs, reflexive coinserters, and effective subcongruences.
		
	\end{defi}
	
	The following result (Corollary \ref{C:main} below) slightly improves and corrects the characterization presented in \cite{ARO}:
	
	\begin{theorem}\label{th1.6}
		An order-enriched category is equivalent to a variety of ordered algebras iff it is subexact and has a subvarietal generator.
	\end{theorem}
	
	\vskip 2mm
	\noindent\textbf{Related Work}\
	Vitale characterized monadic categories over $\Set$ as precisely the finitely complete, exact categories with a regularly projective regular  generator (\cite[Prop. 3.2]{V}). He does not assume finite limits, but uses them all over his proof.
	Our proof, based on Linton's theorem \ref{T:L}, shows that finite limits (beyond kernel pairs) need not be assumed, and a regularly projective strong   generator is sufficient to characterize varieties.

	Similarly, Rosick{\'{y}} and the author characterized classical varieties using the existence of reflexive coequalizers (rather  than just coequalizers of kernel pairs) -- otherwise Corollary 3.6 in \cite{ARV} is the same as Theorem 1.3 above. Thus our theorem is just a tiny improvement, however, precisely that needed for getting the characterization using Barr's exactness. Moreover, the proof we present is simpler than that in \cite{ARV}.
	
	 A closely related result is a recent characterization of varieties of ordered algebras due to Rosick{\'{y}} and the author: in \cite{ARO} subregular epimorphisms and subregular projectives have been introduced, and a characterization theorem was proved that differs from Theorem 1.6 essentially by not working with exactness and by assuming the generator to be subregular. Since small gaps appear in op .cit., we present a corrected version.
	
	Our concept of subcongruence and subexact categories is new. It is related to congruences and exact categories due to Kurz and Velebil \cite{KV} for poset-enriched categories and Bourke and Garner \cite{BG} in general enriched categories, which appear to be quite more technical, however.
	
	\section{Varieties}\label{sec2}
	
	Classical (finitary, one-sorted) varieties were characterized by Lawvere. We explain why a small correction is needed and present some simplifications.
	
	For an object $G$ in $\ck$ with copowers (denoted by $M\cdot G$ for all sets $M$) we obtain the \emph{canonical morphisms}
	$$
	[f]\colon \ck(G,X)\cdot G =\coprod_{f\colon G\to X} G \longrightarrow X \qquad (X\in \obj \ck)\,.
	$$
	Recall that $G$ is a \emph{generator} if all canonical morphisms are epic, a \emph{strong generator} if they extremally epic (do not factorize through a proper subobject of $X$), and a \emph{regular generator} if they are regular epimorphisms. Recall further that $G$ is a \emph{regular  projective} if for each regular
	epimorphism $e\colon X\to Y$  all morphisms from $G$ to $Y$ factorize through $e$. Shortly: $\ck (G,-)$ preserves regular epimorphisms.
	
	Lawvere introduced the following concept; he attributed  it to Freyd.

\begin{defi}[\cite{L}]\label{D:abs} 
	An object $G$ is \emph{abstractly finite} if every morphism from $G$ to a copower $M\cdot G$ factorizes through a finite subcopower.
\end{defi}

That is, for every set $M$ and every morphism $f\colon G\to M\cdot G$  there exists a finite subset $u\colon M_0 \hookrightarrow M$ such that  $f$ factorizes through $u\cdot G\colon M_0 \cdot G \to M\cdot G$ (the morphism  induced by $u$).

\begin{example}
	In $\Set$ this means  that $G$ is finite, in the category of vector spaces that $G$ is finite-dimensional.
	
	Every finite poset is abstractly finite in $\Pos$. But also the linearly ordered set $\R$ is. In fact, every poset with finitely many connected components is abstractly finite.
\end{example}
 \begin{lemma}\label{L:reg}
	Let $\ck$ have kernel pairs and their coequalizers. Every regularly projective strong generator $G$  with copowers is a regular generator.
\end{lemma}

 \begin{proof}
 For every object $X$ let us prove that the morphism
 $$
 [h] \colon \coprod_{h\colon G\to X} G \to X
 $$
 is the coequalizer of its kernel pair $r_0$, $r_1$. Let $e$ be the coequalizer of that pair, and $m$ the unique factorization:
 $$
\xymatrix@C=4pc@R=4pc{
R \ar@<0.5ex>[r]^{r_1}
      \ar@<-0.5ex>[r]_{r_0} &\coprod\limits_{h\colon G\to X}G \ar[d]_e  \ar[r]^{[h]}& X\\
G\ar@<0.5ex>[ur]^{u'_0}
\ar@<-0.5ex>[ur]_{u'_1}
\ar@<0.5ex>[r]^{u_1}
\ar@<-0.5ex>[r]_{u_0}
\ar[u]^v&Y \ar[ur]_m&
}
$$
 Since $[h]$ is a strong epimorphism, so is $m$. Thus it sufficient to prove that $m$ is monic: then it is invertible. Since $G$ is a generator, we can restrict ourselves to parallel pairs with domain $G$.
 
 Given $u_0$, $u_1\colon G\to Y$ with $m\cdot u_0 = m\cdot u_1$, we verify $u_0 = u_1$. Since $G$ is a regular projective, we have $u'_i$ with $u_i = e\cdot u'_i$ ($i=0,1$). From $m\cdot u_0 = m\cdot u_1$ we get $[h]\cdot u'_0 = [h] \cdot u'_1$. Since $r_0$, $r_1$ is the kernel pair of $[h]$, there is $v\colon G\to R$ with $r_i\cdot v=u'_i$ ($i=0,1$). Thus
 $$
 u_0 = e\cdot r_0 \cdot v = e\cdot r_1 \cdot v= u_1\,. 
$$
 \end{proof}

Recall that an object is \emph{finitely generated} if irs hom-hom-functor preserves directed colimits of
monomorphisms.
\begin{lemma}\label{L:fg} 
Let $\ck$ be a cocomplete category with kernel pairs. 

{\rm (1)} Every finitely generated object $G$ is abstractly finite.

{\rm (2)} If $G$ is a varietal generator  (\ref{D:vg}) , then 

\begin{center}
abstractly finite $\Leftrightarrow$ finitely generated
\end{center}
\end{lemma}

\begin{proof}
(1) Just use that for $M$ infinite the copower $M\cdot G$ is the directed colimit of all $M_0\cdot G$ for $\phi \ne M_0\subseteq M$ finite.
The connecting morphisms are split monomorphisms.

(2) Let $G$ be an abstractly finite regular projective
strong generator. We first prove an auxilliary fact:

(a) Every morphism $f\colon G\to \coprod\limits_{i\in I} A_i$ factorizes through a finite subcopower of $\coprod\limits_{i\in I} A_i$. Indeed, each of the canonical morphisms
$$ c_i = [h] \colon \coprod_{h\colon G\to A_i} G\to A_i
$$
is a regular epimorphism (by the preceding lemma).  Thus the morphism
$$
c=\coprod_{i\in I} c_i \colon \coprod_{i\in I} \coprod_{h\colon G\to A_i} G\to \coprod_{i\in I} A_i
$$
is also a regular epimorphism. As $G$ is a regular projective, there exists a factorization of $f$ as $f= c\cdot g$ for some $g\colon G\to  \coprod\limits_{i\in I} \coprod\limits_{h\colon G\to A_i} G$. Since $g$ factorizes through a finite subcoproduct, we have a  finite subset $J \subseteq I$ such that $g$ factorizes through the subcoproduct $\coprod\limits_{i\in J} \coprod\limits_{h\colon G\to A_i} G$. Consequently $f= c\cdot g$ factorizes through $\coprod\limits_{i\in J} A_i$, as claimed.

(b) We prove that $G$ is finitely generated.
 Given a colimit $a_i \colon A_i \to  A$ ($i\in I$) of a directed diagram $D$ of monomorphisms, our task is to prove that $\ck (G, -)$ preserves it. In other words: every morphism $f\colon G\to A$ factorizes through some $a_i$. The standard construction of colimits via coproducts and coequalizers
proves that $[a_i] \colon \coprod\limits_{i\in I} A_i \to A$ is a regular epimorphism. Thus $f$ factorizes through it. Since $G$ is abstractly finite, $f$ factorizes through $[a_j]_{j\in J} \colon \coprod\limits_{j\in J} A_j \to A$ for some finite subset $J\subseteq I$. The diagram $D$ is directed, so we can find an upper bound $i\in I$ of $J$. Then $[a_j]_{j\in J}$  factorizes through $a_i$, thus so does $f$.
\end{proof}

 Recall that regular epimorphisms are \emph{stable under pullback} if in every pullback
 $$
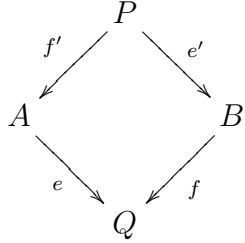
\xymatrix@C=2pc@R=2pc{
& P\ar[dl]_{f'} \ar[dr]^{e'} &\\
A\ar[dr]_{e} && B \ar[dl]^{f}\\
& Q &
}
$$
with $e$ a regular epimorphism, so is $e'$.

\begin{lemma}\label{L:stable}
Let $G$ be a regularly projective strong generator with copowers. If $\ck$ has kernel pairs and their coequalizers, then it has

{\rm(1)} Regular factorizations;

{\rm(2)} Stability of regular epimorphisms under pullback.
\end{lemma}

\begin{proof}
(1) Every morphisms $f\colon A\to B$ with kernel pair $r_0, r_1$  factorizes as $f= m\cdot c$ where $c$ is the coequalizer of  $r_0, r_1$. The proof that $m$ is monic is analogous to the proof of Lemma \ref{L:reg}.

(2) In the above pullback we observe that every morphism $g\colon G\to B$ factorizes through  $e'$:
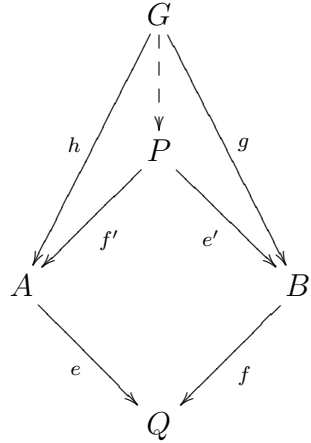
$$
\xymatrix@C=3pc@R=3pc{
& G \ar[ddl]_{h} 
\ar@{-->}[d] \ar[ddr]^{g}&\\
& P\ar[dl]^{f'} \ar[dr]_{e'} &\\
A\ar[dr]_{e} && B \ar[dl]^{f}\\
& Q &
}
$$
 Indeed, since $G$ is a regular projective, for the composite $f\cdot g\colon G\to C$ there is a factorization (say, $h$) through $e$. The universal property of the pullback yields the desired factorization of $g$.

 Let $e'=m\cdot c$  be the regular factorization of $e'$. Then every morphism $g\colon G\to B$ factorizes also through $m$, and since $G$ is a strong generator, this proves that $m$ is invertible. Thus $e'$ is a regular epimorphism.
 \end{proof}

 \begin{remark}\label{R:eq}
 We recall that in a (not necessarily finitely complete) category a \emph{relation} on an object $A$ is represented by a collectively monic pair of morphisms
 $$
 r_0, r_1\colon R\to A\,.
 $$
 
  We say that a parallel pair $s_0$, $s_1\colon S\to A$ \emph{factorizes} through the relation if there is $f\colon S\to R$ with $s_i = r_i\cdot f$ ($i=0,1$).
\end{remark}

\begin{defi}
A \emph{congruence} is a relation $ r_0, r_1\colon R\to A\,$ which is 

(i) \emph{Reflexive:} $r_0, r_1$ are split epimorphisms with a  joint splitting. Equivalently: for every morphism $s\colon S\to A$ the pair $s, s$ factorizes through $r_0, r_1$.

(ii) \emph{Symmetric:} $r_1, r_0$ factorizes through $r_0, r_1$.  Equivalently: if $s$, $s' \colon S\to A$ factorize through $r_0, r_1$, then so do $s', s$.

(iii) \emph{Transitive:} given morphisms $s, s', s'' \colon S\to A$ such that both $s, s'$ and $s', s''$ factorize through $r_0, r_1$, then $s, s''$ also factorizes through $r_0, r_1$.

\end{defi}

\begin{example}
Let $f\colon A\to B$ be a morphism. Its \emph{kernel pair} (which is a universal pair $r_0$, $r_1\colon R\to A$ with $f\cdot r_0 = f\cdot r_1$) is a congruence. A  category has \emph{effective congruences} if every congruence is a kernel pair of some morphism.
\end{example}

\begin{remark}
(1) In the presence of pullbacks transitivity simplifies as follows: given a pullback of $r_1$ and $r_0$:
$$r_1 \cdot r'_0 = r_0 \cdot r'_1$$
then the pair $r_0 \cdot r'_0 , r_1 \cdot r'_1$ factorizes through $r_0 , r_1$.

(2) A parallel pair  $r_0, r_1\colon R\to A\,$ is a congruence iff for every object $S$
the hom-functor $\ck (S,-))$ takes it to a set-theoretical equivalence relation on the set
$\ck (S,X)$. That is, the relation $ \{(r_0 \cdot f, r_1 \cdot f); f \colon S \to R\}$ is reflexive, symmetric, and transitive.

(3) Lawvere worked, for a given object $G$, with a relative concept of reflexivity, symmetry and transitivity: instead of taking an arbitrary object $S$ as above, he restricted it to $G=S$. He then called the relation a \emph{congruence with respect to $G$} if the  set-theoretical relation on $\ck (G,A)$  is an equivalence relation. However, this makes no difference in case $G$ is a regular generator:

\end{remark}

\begin{prop}\label{L:G}
If $G$ is a regular generator with copowers,  then every congruence with respect to $G$ is a congruence.
\end{prop}

\begin{proof}
Let $r_0$, $r_1 \colon R\to A$ be a congruence with respect to $G$. We prove that it is a congruence.

(1) Reflexivity. Let $u_0$, $u_1\colon U\to A$ be a pair with the coequalizer  $[h]\colon \coprod\limits_{h\colon G\to A} G\to A$. Since $r_0$, $r_1$  is reflexive with respect to $G$, each pair $h, h$ factorizes through $r_0,r_1$:  there exists $h'\colon
G\to R$ with $h=r_0 \cdot h' = r_1\cdot h'$. The morphism $[h']\colon 
 \coprod\limits_{h\colon G\to A} G\to R$ merges $u_0$ and $u_1$:
$$
\xymatrix@C=3pc@R=3pc{
U \ar@<0.5ex>[r]^<<<<<<{u_1}
      \ar@<-0.5ex>[r]_<<<<<<{u_0} &\coprod\limits_{h\colon G\to A}G
      \ar[r]^{[h]}
      \ar[d]_{[h']}
& A\\
& R \ar@<0.5ex>[ur]^{r_1}
\ar@<-0.5ex>[ur]_{r_0}
&
}
$$
Indeed, we use that the pair $r_0$, $r_1$ is collectively monic. For $r_0$ we have
$$
r_0 \cdot [h'] \cdot u_i= [r_0 \cdot h'] \cdot u_i = [h]\cdot u_i
$$
which is independent of $i=0,1$. The same holds for $r_1$. Consequently, $[h'] \cdot u_0 = [h'] \cdot u_1$. Therefore $[h']$ factorizes through $[h]$: we have $d\colon A\to R$ with $[h'] = d\cdot [h]$. This is a joint splitting of $r_0$ and $r_1$. Indeed, $r_0 \cdot d= \id$ because $[h]$ is epic and
$$
r_0 \cdot d \cdot [h] = r_0 \cdot [d\cdot h] = r_0 \cdot [h'] = [r_0 \cdot h'] = [h]\,.
$$
Analogously for $r_1$.

\vskip 1mm
(2) Symmetry. Let $u_0$, $u_1\colon U \to \coprod\limits_{h\colon G\to R} G$ be a pair with coequalizer $[h] \colon \coprod\limits_{h\colon G\to R} G \to R$. Then symmetry with respect to $G$ implies that given $h\colon G\to R$ (which is a factorization of the pair $r_0\cdot h$, $r_1\cdot h$ through $r_0, r_1$), there exists $h' \colon G\to R$ factorizing $r_1\cdot h$, $r_0\cdot h$ through $r_0$, $r_1$. Thus  we have the following  commutative squares
$$
\xymatrix@C=3pc@R=3pc{
G\ar[r]^{h'} \ar[d]_{h} & R\ar[d]^{r_0}\\
R\ar[r]_{r_1} & A
}
\qquad
\xymatrix@C=3pc@R=3pc{
G\ar[r]^{h'} \ar[d]_{h} & R\ar[d]^{r_1}\\
R\ar[r]_{r_0} & A
}
$$
The morphism $[h'] \colon \coprod\limits_{h\colon G\to R} G \to R$ merges $u_0$ and $u_1$. This is analogous to (1): for $r_0$ we have 
$$
r_0 \cdot [h'] \cdot u_i= [r_0 \cdot h'] \cdot u_i =
[r_1\cdot h] \cdot u_i= r_1\cdot [h] \cdot u_i
$$
which is independent of $i=0,1$. The same holds for $r_1$.

The morphism $d\colon R\to R$ defined by $[h'] = d\cdot [h]$ is the desired factorization of $r_1$, $r_0$ through $r_0$, $r_1$.
Indeed, $r_0= r_1\cdot d$ because $[h]$ is epic and
$$
r_0\cdot [h] = [r_0 \cdot h] = [r_1\cdot h'] = r_1\cdot[h'] = r_1 \cdot d\cdot [h]\,.
$$
Analogously for $r_1 = r_0\cdot d$.

\vskip 1mm
(3) Transitivity. We are given morphisms $s, s', s''\colon S\to A$ for which factorizations $t$ and $t'$ through $r_0, r_1$ below exist: 
$$
\xymatrix@C=3pc@R=3pc{
& S\ar[dl]_{s} \ar[d]^{t} \ar[dr]^{s'} &\\
A& R \ar[l]_{r_0} \ar[r]^{r_1} & A\\
& S\ar[ul]^{s'} \ar[u]_{t'} \ar[ur]_{s''} &
}
$$
Our task is to find $t'' \colon S\to R$ with
$$ s= r_0\cdot t''\quad \mbox{and} \quad s''= r_1 \cdot t''\,.
$$
Since $r_0, r_1$ is a transitive relation with respect to $G$, the set-theoretical relation $\hat R$ on $\ck(G, R)$ consisting of all pairs $(r_0\cdot h, r_1\cdot h)$ for $h\colon G\to R$ is transitive. Consider an arbitrary morphism $g\colon G\to S$. Due to $t$, the pair $(s\cdot g, s'\cdot g)$ lies in $\hat R$; due to $t'$ the pair $(s'\cdot g, s''\cdot g)$ also lies there. Thus, $(s\cdot g, s''\cdot g) \in \hat R$. Hence for each $g\colon G\to S$ there exists $\bar g \colon G\to R$ with
$$
s\cdot g= r_0 \cdot \bar g\quad \mbox{and} \quad s''\cdot g = r_1\cdot \bar g\,.
$$
Let $u_0$, $u_1 \colon U\to  \coprod\limits_{g\colon G\to S} G$ be a pair with coequalizer
$[g]\colon \coprod\limits_{g\colon G\to S} G \to S$. The morphism $[\bar g]\colon \coprod\limits_{g\colon G\to S} G \to R$ merges $u_0$, $u_1$. Indeed, for $r_0$ we have
$$
r_0 \cdot [\bar g] \cdot u_i = [r_0 \cdot \bar g]\cdot u_i = [s\cdot g] \cdot u_i = s\cdot [g]\cdot u_i
$$
which is independent of $i=0,1$. The same holds for $r_1$. We thus get a morphism
$$
t'' \colon S\to R \quad \mbox{with}\quad  [\bar g] = t'' \cdot [g]\,.
$$
It has the desired properties: $s = r_0 \cdot t''$ follows from
$$
s\cdot[g] = [s\cdot g] = [r_0 \cdot \bar g] = r_0 \cdot [\bar g] = r_0 \cdot t'' \cdot [g]\,.
$$
Analogously for $s'' = r_1\cdot t''$.
\end{proof}

We now recall Barr-exactness. In his paper \cite{B} Barr does not require finite limits: only kernel pairs are included in his definition.

\begin{defi}[\cite{B}]\label{D:Barr}
A category is \emph{exact} if

(1) Kernel pairs and their coequalizers exist.

(2) Congruences are effective.

(3) Regular epimorphisms are stable under pullback.
\end{defi}
We have mentioned in the Introduction the claim in Lawvere's thesis (\cite[Thm. 3.2.1]{L}) that varieties are characterized by having finite limits, effective congruences and a generator with copowers which is an abstractly finite regular projective.  Here is a counter-example.

\begin{example}\label{E:cont}
The following category $\Set^\ast$ is not equivalent to a variety: we add to $\Set$ a formal terminal object $\ast$ (with $\Set (\ast, X) =\emptyset$ for all sets $X$). Then the monomorphism  $1\to \ast$ demonstrates that no object of $\Set^\ast$ is a strong generator. In contrast, free algebras in varieties are strong generators.

The category $\Set^\ast$ has finite limits: $\Set$ is closed under  nonempty limits in $\Set^\ast$. A product $X\times  \ast$ where $X$ is a set is $X$ itself, and there are no  new parallel pairs of distinct morphism in $\Set^\ast$.
Effectivity of congruences in $\Set^\ast$ 
 also follows from this fact. Finally, $1$ is an abstractly finite, regularly projective generator of $\Set$.
 \end{example}

\begin{remark}
As observed in \cite{A} another source of counter-examples are non-complete lattices with a top element.
\end{remark}

As mentioned in the Introduction, Lawvere proved in \cite{L} Theorem \ref{th1.1}. Several authors presented various simplifications. For example Pedicchio and Wood \cite{PW} showed that effective congruences can be deleted in case the hom-functor of thegenerator in (3) is assumed to  preserve reflexive coequalizers. This has led to the following

\begin{defi}[\cite{ARV}]
An object is \emph{effective} if its hom-functor preserves coequalizers of congruences.
\end{defi}

\begin{prop}\label{P:eff}
Let $\ck$ be a category with kernel pairs and their coequalizers. For every regularly projective strong generator $G$ we have the equivalence
\begin{center}
$G$ effective $\Leftrightarrow  \ck$ has effective congruences.
\end{center}
\end{prop}

\begin{proof}
(1) Let $\ck$ have effective congruences. Given a regular epimorphism $c\colon A\to C$ and  its kernel pair $r_0, r_1\colon R\to A$, our task is to prove that the map
$$
c\cdot(-) \colon \ck(G,A) \to \ck (G,C)
$$ 
is a coequalizer of $r_i \cdot (-)$ for $i=0,1$. Since $G$ is a regular generator (Proposition \ref{L:reg}), the map $c\cdot (-)$ is a regular epimorphism. Thus, we only need to verify that it has the kernel pair $r_i\cdot(-)$. Indeed, let $c\cdot(-)$ merge a pair in $\ck(G, A)$, say, $c\cdot f_0 = c\cdot f_1$. Then there is a unique $f'\colon G\to R$ with $f_i = r_i \cdot f'$ ($i=0,1$).

\vskip 1mm
(2) Suppose that $\ck (G, -)$ is effective. Let $r_0, r_1\colon R\to A$ be a congruence. Since $\ck(G,-)$ is faithful and preserves pullbacks, the pair  $\ck(G, r_0)$, $\ck(G, r_1)\colon \ck(G,R) \to \ck(G,A)$ is a congruence in $\Set$. We know that the coequalizer $c\colon A\to C $ of $r_0, r_1$ yields a coequalizer $\ck(G,c)$ of $\ck(G, r_i)$. It follows that the above  pair  is a kernel pair of $\ck(G,c)$.

To verify that $r_0$, $r_1$ is the kernel pair of $c$,  be $u_0$, $u_1 \in \ck(G,A)$ fulfil $c\cdot u_0 = c\cdot u_1$. Since the relation  of all $(r_0 \cdot v, r_1 \cdot v)$ for $v \colon G \to R$ is an equivalence,
and $c \cdot (-)$ is its quotient map, there is a unique $v \in \ck(G,R)$ with $u_i=r_i\cdot v$ ($i=0,1$).
\end{proof}

We now prove the main result of the present section. We use the monadicity theorem of Linton:

\begin{theorem}[{\cite[Prop. 3]{Li}}]\label{T:L}
A functor $U\colon \ck \to \Set$ is monadic iff

{\rm(a)} $U$ is right adjoint.

{\rm(b)} $\ck$ has kernel pairs  and coequalizers of congruences.

{\rm(c)} $U$ preserves and reflects congruences.

{\rm(d)} $U$ preserves and reflects regular epimorphisms.
\end{theorem}

\begin{theorem}\label{T:main}
A category is equivalent to a variety iff it is exact and has a varietal generator.
\end{theorem}

\begin{proof}
Necessity. Every variety $\cv$ is well known to be a cocomplete and  exact category. Its free algebra $G$ on one generators an abstractly finite object since it is finitely generated (Lemma \ref{L:fg}). It is a regular  projective: regular epimorphisms are precisely the surjective homomorphisms and $\cv(G,-)$ is naturally isomorphic to the forgetful functor. Finally, $G$ is a strong generator since its copowers are  the free algebras of $\cv$.

Sufficiency. Let $\ck$ be an exact category and $G$  be a varietal generator. For the hom-functor
$$
U =\ck(G,-) \colon \ck \to  Set\,
$$
we prove that it is monadic and the corresponding monad is finitary. Consequently, $\ck$ is equivalent to a variety.

(1) $U$ is monadic. Indeed, $U$ has the left adjoint $M\mapsto M\cdot G$.

We thus only need to verify (c) and (d) in Linton's theorem.

(c1) $U$ preserves congruences. In fact, let $r_0$, $r_1\colon R\to A$ be a congruence. Since $U$ is faithful, $Ur_0$, $Ur_1$ is collectively monic. The relation $Ur_0$, $Ur_1$ in $\Set$ represents the set-theoretical relation $\hat R$ on $\ck(G,R)$ defined  by
$$
\hat R =\big\{(r_0\cdot g, r_1\cdot g); g\colon G\to R\big\}\,.
$$
Since $r_0$, $r_1$ is reflexive, so is $\hat R$: given $d\colon A\to R$ with $r_0\cdot d=\id=r_1\cdot d$, we have, for each $h\colon G\to A$
$$
(h,h) = (r_0 \cdot d\cdot h, r_1\cdot d\cdot h)\in \hat R\,.
$$
Analogously, $\hat R$ is symmetric. To verify  transitivity, let $(r_0\cdot g, r_1\cdot g)$ and 
$(r_0\cdot g', r_1\cdot g')$ be members of $\hat R$ with $r_1\cdot g = r_0\cdot g'$. The pair $r_0\cdot g$, $r_1\cdot g$ also factorizes through $r_0, r_1$ via $g$, and the pair $r_0\cdot g, r_1\cdot g'$ factorizes via $g'$. Since  $r_0, r_1$ is transitive, the pair $r_0\cdot g, r_1\cdot g'$ factorizes through $r_0, r_1$: we have $g''$ with
$$
r_0 \cdot g= r_0 \cdot g''\quad \mbox{and} \quad r_1\cdot g'=r_1\cdot g''\,.
$$
This proves $(r_0\cdot g, r_1\cdot g') \in \hat R$, as desired.

(c2) $U$ reflects congruences. Let $r_0, r_1\colon R\to A$ be a pair such that $U r_0, Ur_1$ is a congruence. Since $G$ is a generator, the fact that $Ur_i = r_i\cdot (-)$ is a collectively monic pair for $i=0,1$ implies that $r_0, r_1$ is collectively monic. To say that $Ur_0, Ur_1$ is a congruence means that $r_0, r_1$ is a congruence with respect to $G$ (Remark \ref{R:G}).

Since $G$ is a regular generator (Proposition \ref{L:reg}) the proof follows from Lemma \ref{L:G}.

(d1) $U$ preserves regular epimorphisms because $G$ is a regular projective.

(d2)  $U$ reflects regular epimorphisms. That is, given a morphism $e\colon A\to B$ such that every morphism $g\colon G\to B$ factorizes through it, we verify that $e$ is a coequalizer of its kernel pair $r_0, r_1\colon R\to A$.

Let $c\colon A\to C$ be a coequalizer of $r_0, r_1$, and let $h$ make the triangle below commutative:
$$
\xymatrix@C=4pc@R=3pc{
& G\ar[dl]_{v} \ar@<0.5ex>[d]^{v_1}
\ar@<-0.5ex>[d]_{v_0}
\ar@<0.5ex>[r]^{u_1}
\ar@<-0.5ex>[r]_{u_0} & C \ar[d]^{h}\\
R \ar@<0.5ex>[r]^{r_1}
      \ar@<-0.5ex>[r]_{r_0} &
A \ar[r]_{e} \ar[ur]^{c} & B
}
$$
 We prove that $h$ is an isomorphism, thus, $e=\coeq (r_0, r_1)$. Every morphism $g\colon G\to B$ factorizes through $e$, hence also through $h$. Thus to verify that $h$ is invertible, it is sufficient to prove that it is monic (using that $G$ is a strong generator). Indeed, for every pair $u_0, u_1\colon G\to C$ with
 $$
 h\cdot u_0 = h\cdot u_1
 $$
 we derive $u_0=u_1$ Since $c$ is a regular epimorphism, we have  $v_i$ with $u_i=c\cdot v_i$. We derive that
 $$
 e\cdot v_0 = h\cdot c\cdot v_0 = h\cdot c \cdot v_1 = e\cdot v_1\,.
 $$
 Therefore there is $v\colon G\to R$ with $v_i = r_i \cdot v$. Thus
 $$
 u_i = c\cdot v_i = c\cdot r_i\cdot v
 $$
 is independent of $i=0,1$.

(ii) The functor $T=UF$, where $F$ is the left adjoint of $U$, is finitary because $G$ is finitely generated (Lemma \ref{L:fg}). Indeed, $F$ preserves directed colimits of nonempty monomorphisms and these monomorphisms split. Consequently, $T=\ck(G,-)\cdot F$ preserves these colimits, too. Given an infinite set $X$, express it as the directed colimit of all of its finite nonempty subsets. Since $T$ preserves this colimit, for every element $x\in TX$ there exists a finite subset $m\colon M\hookrightarrow X$ such that $x$ lies in $Tm[TM]$. By \cite[Thm 3.4]{AMSW}, this implies that $T$ is finitary.
 \end{proof}
  
 Observe that we have not used the stability of regular epimorphisms under pullback in the above proof. (No surprise -- see Lemma \ref{L:stable}.)  We thus get, using Proposition \ref{P:eff}, the following statement slightly improving Corollary 36 of \cite{ARV}.
 \begin{corollary}
 A category is equivalent to a variety iff it has 
 
{\rm (1)} Kernel pairs and their coequalizers.

{\rm (2)} An effective, abstractly finite, strong  generator with copowers.
\end{corollary}

\section{Reflexive Coequalizers}\label{sec3}

Before turning to order-enriched varieties in Section \ref{sec4}, we prove an auxiliary proposition for enriched categories in general.  In the present section we assume that a symmetric monoidal closed category
$$
(\cv, \otimes, I)
$$
is given (which in Section \ref{sec4} will be the cartesian closed category of posets).

\begin{remark}\label{re3.1}
Let $\ck$ be an enriched category. When speaking about ordinary colimits (coproducts, coequalizers, etc.) we always mean the conical ones: weighted colimits with the weight constant with value $I$.

\emph{Reflexive coequalizers} are (conical) coequalizers of pairs $r_0, r_1\colon R\to X$ that are reflexive: there is $d\colon X\to R$ with $r_i\cdot d= \id_X$.
\end{remark}

\begin{defi}[{\cite{K}}] 
A full subcategory $\ca$ of an enriched category $\ck$ is \emph{dense} if the functor
$$
E\colon \ck \to [\ca^{\op}, \cv]
$$
assigning to $K\in \ck$ the restriction of $\ck(-, A)$ to $\ca^{\op}$ is fully faithful.
\end{defi}

\begin{remark}\label{R:ten}
 Recall that an object $G$ of an enriched category  \emph{has tensors} if the hom-functor $\ck)G, -) \colon \ck \to \cv$ has a  left adjoint $F$.  The notation is $P\otimes G$ for $FP$.
 \end{remark}
 
\begin{prop}\label{P:refl} 
Let $\ck$ be an enriched category with reflexive coequalizers and $\ca$ be a small full dense subcategory such that

(1) Objects of $\ca$ have tensors  in $\ck$. 

(2) $\ck$ has coproducts of collections of such tensors.
Then $\ck$ it is equivalent to a full reflective subcategory of $[\ca^{\op}, \cv]$.
\end{prop}

\begin{proof}
Since the functor $E\colon \ck \to [\ca^{\op}, \cv]$ is fully faithful, we only need to prove that it has a left adjoint. That is, $E[\ck]$ is a reflective subcategory. The reflection of an object $H\colon \ca^{\op}\to \cv$ is given by an object $H^\ast$ of $\ck$ and a natural transformation $\varrho \colon H\to EH^\ast$ we construct now.

We first form coproducts
$$
X= \coprod_{A\in \obj \ca} HA \otimes A \quad \mbox{and} \quad Y= \coprod_{f\colon B\to A}HA \otimes B
$$
with injections
$$
i(A) \colon HA\otimes A\to X \quad \mbox{and}\quad j(f) \colon HA \otimes B\to Y\,.
$$
Every morphism $f\colon B\to A$ of $\ca$ yields a parallel pair $p_f, q_f \colon HA \otimes B\to X$ as follows
$$
\xymatrix@=3pc{
& HA \otimes B \ar[dl]_{HA\otimes f} \ar@<0.5ex>[d]_{p_f\ }
\ar@<-0.5ex>[d]^{\ g_f }
\ar[dr]^{Hf \otimes B}&\\
HA\otimes A \ar[r]_{i(A)} & X & HB\otimes B \ar[l]^{i(B)}
}
$$
The resulting pair $[p_f], [q_f] \colon Y\to X$ is reflexive: both morphisms are split by the morphism
$$
\big[ j(\id_{\ca})\big] \colon X\to Y\,.
$$
The desired object $H^\ast$ of $\ck$ is given by the following coequalizer
$$
\xymatrix@=3pc{
Y \ar@<-.5ex>[r]_{[p_f]}
\ar@<.5ex>[r]^{[q_f]}
& X \ar[r]^{c} & H^\ast
}
$$
The components $\varrho_A \colon HA \to EH=\ck (A, H^\ast)$ are given by their adjoint transposes $c\cdot i(A)$:
$$
\xymatrix@R=.21pc{
&HA \ar[rr]^{\varrho_A}  && \ck (A, H^\ast)&\\
\ar@{-}[rrr]&&&\\
&HA \otimes A \ar[r]_<<<<{i(A)} &X \ar[r]_{c}&&
}
$$
Let us verify the naturality of $\varrho$. Given $f\colon B\to A$ in $\ca$ the square below
$$
\xymatrix@=2pc{
HA \otimes B \ar[dd]_{Hf\otimes B} \ar[rr]^{HA \otimes f}
&&
 HA \otimes A \ar[dd]^{\varrho_A}
              \ar@{-->}[dl]_<<<<<<<<<{i(A)}\\
 & X \ar@{-->}[dr]^{c}& \\
 HB \otimes B \ar[rr]_{\varrho_B}
    \ar@{-->}[ur]^{i(B)}&& H^\ast
    }
    $$
    commutes due to $c\cdot p_f = c\cdot  q_f$. Its adjoint transpose is the desired naturality  square
$$
\xymatrix@=3pc{
HA \ar[r]^{\varrho_A}\ar[d]_{Hf} & \ck(A, H^\ast)\ar[d]^{(-)\cdot f}\\
HB \ar[r]_{\varrho_B} & \ck(B, H^\ast)
}
$$ 

Let us prove the universal property of $\varrho \colon H\to EH^\ast$. Let an object $K\in \ck$ and a morphism $\sigma \colon H\to EK$ (a natural transformation) be given. Then the adjoint transposes of $\sigma_A \colon HA \to \ck (A,K)$
$$
\widehat \sigma_A \colon HA \otimes A \to K
$$
make the following  squares commutative:
$$
\xymatrix@=2pc{
HA \otimes B \ar[dd]_{Hf\otimes B} \ar[rr]^{HA \otimes f}
&&
 HA \otimes A \ar[dd]^{\widehat\sigma_A}
              \ar@{-->}[dl]_<<<<<<<<<{i(A)}\\
 & X \ar@{-->}[dr]^{c}& \\
 HB \otimes B \ar[rr]_{\widehat\sigma_B}
    \ar@{-->}[ur]^{i(B)}&& K
    }
    $$
Therefore, the parallel pair $[p_f], [q_f]$ is merged by
$$
\big[ \widehat\sigma_A\big] \colon X= \coprod_A HA \otimes A \to K\,.
$$
Consequently, there is a unique morphism  $\sigma^\ast \colon H^\ast \to K$ in $\ck$ making the triangle below commutative
$$
\xymatrix@C=3pc@R=2pc{
&& H^\ast \ar[d]^{\sigma^\ast}\\
HA \otimes A \ar@{-->}[r]_{i(A)} & X \ar[ur]^{c} \ar[r]_{[\widehat\sigma_A]} & K
}
$$
Which is equivalent to $\sigma = E\sigma^\ast \cdot \varrho$: indeed, for every $A\in \ca$ we have $$\sigma_A = \sigma^\ast \cdot \varrho_A \Leftrightarrow \ \widehat \sigma_A =\sigma^\ast \cdot \widehat \varrho_A,$$
and this means precisely the commutativity of the above triangle. Thus, $\varrho\colon H\to EH^\ast$ defines a reflection of $H$ in $[\ca^{\op}, \cv]$.
\end{proof}

\begin{corollary}\label{C:refl}
Let $\ck$ have reflexive coequalizers and an object $G$ with tensors such that all finite copowers form a dense full subcategory. If $\cv$ is (co)complete, then so in $\ck$.
\end{corollary}

Indeed, the full subcategory $\ca$ of all finite copowers of $G$ satisfies (1) and (2) of the above proposition: for (1) use $P\otimes \big(\coprod\limits_{n} G\big) = \big(\coprod\limits_{n} P\big) \otimes G$. Analogously for (2). If $\cv$ is (co)complete, so is $[\ca^{\op}, \cv]$ (\cite{K}, Section 3.3.3), and since $\ck$ is equivalent to a full reflective subcategory, it is also (co)complete (\cite{K}, Section 3.3.5).

\section{Varieties of Ordered Algebras} \label{sec4}

Here we present a characterization of varieties of ordered algebras analogous to Theorem \ref{T:main}. This follows ideas of \cite{ARO} endowed with the concept of a subexact category introduced below.

\begin{nota} \label{N:free}
(1) Let $\Sigma = (\Sigma_n)_{n<\omega}$ be a signature. We denote by
$$
\Sigma \mbox{-} \Pos
$$
the category of ordered $\Sigma$-algebras with monotone operations as objects, and monotone homomorphisms as morphisms.

(2) Every set is considered as a (discretely ordered) poset. The free $\Sigma$-algebra $T_{\Sigma} X$ on a set $X$ (of all terms in variables from $X$), discretely ordered, is also a free ordered $\Sigma$-algebra. Given an ordered $\Sigma$-algebra $A$ and a map $h\colon X \to A$, we denote by
$$
h^\# \colon T_\Sigma X\to A
$$
the corresponding homomorphism.
\end{nota}

\begin{defi}
 A \emph{variety of ordered algebras} is a full subcategory of $\Sigma$-$\Pos$ presented by inequations $t\leq s$ between terms $t$, $s\in T_{\Sigma} X$. It consists of algebras $A$ such that $h^\#(t) \leq h^\# (s)$ holds for each of the inequations and each interpretation $h\colon X\to A$ of the variables.
\end{defi}

\begin{remark}
The category $\Pos$ is cartesian closed with $[X,Y] =\Pos(X,Y)$ ordered pointwise ($f\leq g\colon X\to Y$ means $f(x) \leq g(x)$ for each $x\in X$). Categories enriched over it are called  order-enriched. This means a category endowed with partial order on each hom-set such that composition ismonotone. Given order-enriched categories $\ck$ and $\cl$, a functor $F\colon \ck \to \cl$ is enriched iff it is locally monotone: given $f\leq g$ in $\ck(X,Y)$, we have $Ff\leq Fg$ in $\cl(X,Y)$.
\end{remark}

\begin{example}
$\Sigma$-$\Pos$ is enriched with the pointwise order on hom-sets. Every variety is thus also enriched.
\end{example}

Whereas coequalizers and regular epimorhisms  play a central role in the characterization of  classical varieties, the corresponding role is taken by coinserters and subregular epimorphisms in $\Sigma$-$\Pos$.

\begin{defi}\label{D:coins}
Let $f_0$, $f_1\colon X\to Y$ be morphisms of an order-enriched category. (We use  indices $0$, $1$ to indicate that $f_0$ comes first and $f_1$ second. We do \emph{not} assume $f_0\leq f_1$ in $\ck(X,Y)$.)

Their \emph{coinserter} is the universal morphism $c\colon Y\to Z$ with respect to $c\cdot f_0 \leq c\cdot f_1$ in $\ck(X,Y)$. That is: 
\begin{enumerate}
\item[(1)] Every morphism $c'\colon Y\to Z'$ with $c'\cdot f_0\leq c\cdot f_1$ factorizes through~$c$.

\item[(2)] Given $u_0$, $u_1\colon Z\to U$ with $u_0\cdot c\leq u_1\cdot c$, it follows that $u_0\leq u_1$.
\end{enumerate}
\end{defi}

\begin{example}
In $\Pos$ the coinserter of $f_0$, $f_1\colon X\to Y$ is given as follows.
Recall that a \emph{preorder} is a reflexive and transitive relation. The \emph{posetal reflection} of a preordered set $(Y, \sqsubseteq)$ is the quotient modulo the equivalence $\sim$ with $y\sim y'$ iff $y\sqsubseteq y' \sqsubseteq y$.

Let $\sqsubseteq$ be the  least preorder on $Y$ with $f_0(x) \sqsubseteq f_1(x)$ (for all $x\in X$) and containing the order of $Y$. The coinserter $c \colon (Y, \leq) \to C$ of $f_0, f_1$ is given by the posetal reflection 
$$
c\colon (Y, \sqsubseteq ) \to (Y, \sqsubseteq )/ \sim =C.
$$

\end{example}

\begin{defi}[\cite{ARO}]\label{D:sub}
A morphism $c\colon Y\to Z$ in an order-enriched category is a \emph{subregular epimorphism} if  it is a coinserter of a reflexive parallel pair  $f_0$, $f_1$:
$$
\xymatrix@C=4pc@R=3pc{
X \ar@<0.5ex>[r]^{f_1}
\ar@<-0.5ex>[r]_{f_0}
&
Y \ar@/_2pc/[l]_{d} \ar[r]^{c} &
Z
}\qquad f_0\cdot d=\id_Y = f_1\cdot d\,.
$$
\end{defi}

\begin{example} 
Subregular epimorphisms in $\Pos$, and more generally in $\Sigma$-$\Pos$, are precisely the surjective homomorphisms (\cite{ARO}, Prop. 4.4).

(2) If an order-enriched category has finite coproducts, then we have
\begin{center}
regular epi $\Rightarrow$ subregular epi $\Rightarrow$ epi
\end{center}
(\cite{ARO}, Ex. 3.4).

\end{example}

\begin{defi}
Let $\ck$ be an order-enriched category.

\vskip1mm
(1) A \emph{relation} on an object $A$ is a parallel pair $r_0$, $r_1\colon R\to A$ which is \emph{collectively order-reflecting}: whenever morphisms $f$, $f'\colon X \to R$ fulfil $r_0\cdot f \leq r_0\cdot f'$ and $r_1\cdot f \leq r_1\cdot f'$, then $f\leq f'$.

\vskip 1mm
(2)
 A \emph{subkernel pair} of a morphism $h\colon A\to B$ is a universal parallel pair $r_0$, $r_1\colon R\to A$ with respect to $h\cdot r_0 \leq h\cdot r_1$.

That is, a relation on $A$ such that every pair $v_0$, $v_1\colon V\to A$ with $h\cdot v_0 \leq h\cdot v_1$ factorizes through $r_0$, $r_1$.
\end{defi}
  
In ordinary category theory the concept of congruence is an abstraction of kernel pairs: every  kernel pair is a congruence, and the opposite  implication holds in $\Set$ (and other categories, e.g.\ varieties). In order-enriched categories we introduce subcongruences which are abstractions of subkernel pairs. Each subkernel pair is reflexive (even hyper-reflexive, see below) and transitive. It is, of course, not symmetric.

Recall from Remark \ref{R:eq} that the  reflexivity of $r_0$, $r_1\colon R\to A$ means that of all $s\colon S\to A$ the pair $s$, $s$ factorizes through $r_0$, $r_1$. Here is a stronger property:

\begin{defi}\label{D:hyper}
A relation $r_0$, $r_1\colon R\to A$ in an order-enriched category is
\emph{order-reflexive} if every comparable pair $s_0 \leq s_1 \colon S\to A$ factorizes through $r_0$, $r_1$.
\end{defi}

Every subkernel pair is order-reflexive: from $s_0 \leq s_1$ it follows that $f \cdot s_0 \leq f \cdot s_1$,
thus $s_0 ,s_1$ factorize through the subkernel pair of $f$.

\begin {defi} A \emph{subcongruence} is a hyper-reflexive and transitive relation.
\end{defi}

Thus every subkernel pair is  a subcongruence.
In all varieties we will prove the reverse implication: 
every subcongruence is a subkernel pair.To achieve this, we first show how coinserters of congruences are constructed in $\Pos$.

\begin{constr}\label{C:ref}
For every  subcongruence $r_0, r_1\colon R\to A$ in $\Pos$ the following relation on $A$ is a preorder:
$$
x \sqsubseteq y \mbox{\quad iff \quad} x= r_0(z) \quad \mbox{and}\quad y=r_1(z) \quad \mbox{for some\quad} z\in R.
$$
The posetal reflection $c\colon (A, \sqsubseteq) \to C$ yields the coinserter $c\colon A\to C$ of $r_0$ and $r_1$.
\end{constr}

\begin{proof}
Let $\leq$ denote the given partial order on $A$. We verify that $\sqsubseteq$ is indeed a preorder, and that it contains $\leq$. Thus $c$ is a monotone map from $(A, \leq)$ to $C$. It then easily follows that $c$ is the coinserter of $r_0$ and $r_1$.

(1) The relation $\sqsubseteq$ is reflexive because $r_0$, $r_1$ is a reflexive relation. To prove that $\sqsubseteq$  is a preorder, we verify the transitivity:
\begin{center}
if \ $x\sqsubseteq x' \sqsubseteq x''$ \quad then \quad $x\sqsubseteq x''$.
\end{center}
We are given $z$, $z'\in R$ with
$$
x= r_0(z)\,, \ x'=r_1(z) = r_0(z')\quad \mbox{and}\quad x''=r_1(z')\,.
$$
Let $s, s', s''\colon 1\to A$ be the morphisms representing $x$, $x'$ and $x''$, resp. Then $s$, $s'$ factorizes through $r_0, r_1$: use the morphism $1\to R$ representing  $z$.
Analogously, $s', s''$ factorizes through $r_0, r_1$.
 Since  the relation $r_0, r_1$ is transitive, $s, s''$ also factorize through $r_0, r_1$. The factorizing morphism represents  an element $z''\in R$ such that $x=r_0(z)= r_0(z'')$ and  $x''=r_1(z')= r_1(z'')$. This verifies that $x\sqsubseteq x''$.

(2) We show that
$$
x_0\leq x_1 \quad \mbox{implies}\quad x_0 \sqsubseteq x_1\,.
$$
We have morphisms $q_i \colon 1\to A$ which represent $x_i$ 
($i=0,1$). Then $q_0 \leq q_1$, thus by hyper-reflexivity there exists $k\colon 1\to R$ with  $q_0 = r_0\cdot k$ and $ q_1 = r_1 \cdot k$. In other words, the element $z\in R$ represented by $k$ fulfils $x_0 = r_0(z)$ and $x_1= r_1(z)$; hence $x_0 \sqsubseteq x_1$.

(3) The monotone map $c\colon (A, \leq ) \to C$ is a coinserter of $r_0$ and $r_1$. In fact, $c\cdot r_0 \leq c\cdot r_1$: for $z\in R$ we have  $r_0(z) \sqsubseteq r_1(z)$, thus $c\cdot r_0(z) \leq c \cdot r_1(z)$.

Let $c'\colon A\to C'$ fulfill $c'\cdot r_0 \leq c'\cdot r_1$. To prove that $c'$ factorizes through the posetal reflection $c$ of $(A, \sqsubseteq)$, we just need to verify that
$$
x_0 \sqsubseteq x_1 \quad \mbox{implies}\quad c'(x_0) \leq c'(x_1) \quad \mbox{in}\quad C'\,.
$$
But this follows trivially from $c'\cdot r_0 \leq c'\cdot r_1$.
\end{proof}

\begin{defi}
An order-enriched category has \emph{effective subcongruences} if every subcongruence is the subkernel pair of some morphism.
\end{defi}

\begin{prop}\label{P:subex}
The category of ordered $\Sigma$-algebras has effective subcongruences.
\end{prop}
\begin{proof}
(1) $\Pos$ has effective subcongruences. Indeed, given a subcongruence $r_0, r_1\colon R\to A$ and the morphism $c\colon A\to C$ of the above construction, then $r_0, r_1$ is the subkernel pair of $c$: First, $c\cdot r_0 \leq c \cdot r_1$ clearly holds. Second, every morphism $c'\colon A\to C'$ with $c'\cdot r_0 \leq c'\cdot r_1$ fulfils the implication
$$
x_0 \sqsubseteq x_1 \quad \mbox{implies}\quad c'(x_0) \leq c'(x_1) \quad (x_0, x_1 \in A)\,.
$$
 Thus $c'$ factorizes through the posetal reflection of $\sqsubseteq$, which is $c$. Since, moreover, $c$ is surjective, it has required universal property.
 
 (2) For every $n\in \N$ the morphism $c^n \colon A^n\to C^n$ has the  subkernel pairs $r_0^n, r_1^n$. This follows easily from the above construction.
 
 (3) We are ready to prove that $\Sigma$-$\Pos$ has effective  subcongruences. Let $r_0, r_1\colon R\to A$ be homomorphisms forming a subcongruence in  $\Sigma$-$\Pos$. The forgetful functor $U \colon \Sigma$-$\Pos \to \Pos$ preserves subcongruences. In fact, let $r_0, r_1$ be a subcongruence in $\Sigma$-$Pos$.
 
 a. $Ur_0, Ur_1$ is transitive because $U$ preserves pullbacks (in fact, it creates limits).
 
 b. $Ur_0, Ur_1$ is order-reflexive: let $s_0 \leq s_1 \colon S \to UA$ be given. The corresponding homomorphisms $ s_i ^\sharp \colon F_\Sigma S \to A$ are also comparable, thus they factorize through $r_0, r_1$ in $\Sigma$-Pos. Consequently 
 $s_i = Us_i ^\sharp \cdot \eta_S $ implies that $s_0, s_1$ factorize through $Ur_0, Ur_1$.
 
  By Item (1) $Ur_0, Ur_1$ is the kernel pair  of a (surjective) morphism $c\colon A\to C$. We prove that $C$ carries a unique structure of an algebra making $c$ a homomorphism. In other words, for every $n$-ary operation symbol $\sigma\in \Sigma$ a unique morphism $\sigma_C$ exists making the square below commutative:
$$
\xymatrix@C=2pc@R=2pc{
R^n \ar@{-->}@<0.5ex>[r]^{r_1^n}
\ar@{-->}@<-0.5ex>[r]_{r_0^n}
& A^n \ar[d]_{c^n} \ar[r]^{\sigma_A} & A \ar[d]^{c}\\
& C^n \ar[r]_{\sigma_C} & C
}
$$
Indeed, by (2), $r_0^n$, $r_1^n$ is a subcongruence on $A^n$, and $c^n$ is the coinserter. Since $r_0, r_1$ are homomorphisms, we have $(c\cdot \sigma_A)\cdot r_0 ^n \leq (c\cdot \sigma_A)\cdot r_1 ^n$, indeed:
$$
 c\cdot \sigma_A\cdot r_0 ^n = c\cdot r_0 \cdot\sigma_R \leq c\cdot r_1 \cdot \sigma_R = c\cdot \sigma_A \cdot r_1 ^n\,.
$$
Thus we get the unique $\sigma_C$ as stated.

Moreover, the homomorphism $c$ is the coinserter of $r_0$ and $r_1$ in $\Sigma$-$\Pos$. Indeed, given a homomorphism $c'\colon A\to C'$ with $c'\cdot r_0\leq c'\cdot r_1$, there is a unique monotone map $h$ making the triangle below commutative in $\Pos$:
$$
\xymatrix@C=1pc@R=2pc{
& A\ar[dl]_{c} \ar[dr]^{c'} & \\
C \ar[rr]_{h} && C'
}
$$
Since $c$ and $c'$ are homomorphisms and $c$ is surjective, it follows that $h$ is also a homomorphism. Thus, $c$ is the coinserter of $r_0$ and $r_1$ in $\Sigma$-$\Pos$.
\end{proof}

In the classical universal algebra Birkhoff's Variety Theorem states that a full subcategorry of $\Sigma$-$\Alg$ is a variety iff if is closed under products, subalgebras, and quotients (= homomorphic images). For ordered algebras we have the analogous three constructions:

(1) A product of algebras $A_i$ ($i\in I$) is their cartesian product with both operations and order given coordinate-wise.

(2) By a \emph{subalgebra} of an ordered algebra $A$ is meant a subposet closed under the operations. Thus subalgebras are represented by homomorphisms $m\colon B\to A$ such that for $x$, $y\in B$ we have $x\leq y$ iff $m(x) \leq m(y)$.

(3) By a \emph{homomorphic image} of an algebra we  mean a quotient represented by a subregular epimorphism $e\colon A\to B$. (That is, $e$ is surjective.)

\begin{birk}[\cite{ADV}]
 A full subcategory of $\Sigma$-$\Pos$ is a variety of ordered algebras iff it is closed under products, subalgebras, and homomorphic images.
 \end{birk}
 
 \begin{corollary}\label{C:eff}
 Every variety of ordered algebras has effective subcongruences.
 \end{corollary}
 
 Indeed, given a variety $\cv \subseteq \Sigma$-$\Pos$, since it is closed under pullbacks (in fact, under limits), every subcongruence $r_0, r_1\colon R\to A$ in $\cv$ is also a subcongruence  in    $\Sigma$-$\Pos$. By Proposition~\ref{P:subex} there is a homomorphism  $h\colon A\to B$ with subkernel pair $r_0, r_1$. From the proof of that proposition we know  that $h$ is surjective. Hence $B$ is a homomorphic image of $A\in \cv$. Consequently, $B\in \cv$ and $h$ is a morphism of $\cv$ with the subkernel pair $r_0, r_1$.

\begin{defi}[\cite{ARO}]
An object $G$ of an  order-enriched category is a \emph{subregular projective}  if it is hom-functor to $\Pos$ preserves subregular epimorphisms. That is, given a subregular epimorphism $e\colon A\to B$, every morphism from $G$ to $B$ factorizes through it:
$$
\xymatrix@C=2pc@R=2pc{
& G \ar@{-->}[dl]_{\exists} \ar[d]^{\forall}\\
A\ar[r]_{e} & B
}
$$
\end{defi}

\begin{example}[{\cite[Ex. 4.6]{ARO}}]\label{E:nec}
The free algebra $G$ on one generator in a variety $\cv$ is an abstractly finite subregular projective.

Finally, $G$ is a \emph{strong generator} in the enriched sense (Kelly \cite{K}): the hom-functor $\cv (G,-)\colon \cv \to \Pos$ reflects isomorphisms. Indeed, given a homomorphism $h\colon A\to B$ such that the morphism
$$
h\cdot(-)\colon \cv (G, A) \to \cv(G,B)
$$
is invertible in $\Pos$, it follows that $h$ is a bijection which preserves and reflects the order. Consequently, $h^{-1} \colon B\to A$ is a monotone  homomorphism.
\end{example}

\begin{lemma} \label{L:ten}
In a category with reflexive coinserters every object $G$ with copowers has tensors.
\end{lemma}

\begin{proof}
We describe, for every poset $P$, the tensor $C= P\otimes G$ as the following reflexive coinserter:
$$
\xymatrix@C=2pc{
R\cdot G \ar@<0.5ex>[r]^{\bar r_1}
\ar@<-0.5ex>[r]_{\bar r_0} &|P| \cdot G \ar[r]^{c} & C
}
$$
Here $|P|$ is the underlying set of $P$ and $R\subseteq |P| \times |P|$ is its order relation. The morphisms $\bar r_i$ are  induced by the projection $r_i \colon R\to |P|$ given by $r_i (x_0, x_1) = x_i$. The diagonal $\Delta \colon |P| \to R$ yields a joint splitting $\Delta \cdot G$ of the pair $\bar r_i=r_i\cdot G$, thus, the coinserter exists. Its components are denoted by $c_x \colon G \to C$ for $ x \in P$.

Our task is to find a natural isomorphism
$$
\xymatrix@C=2pc@R=.21pc{
&C\ar[r]^{f} &X&\\
\ar@{-}[rrr]&&&\\
&P\ar[r]_<<<<<{i(f)} &\ck(G,X)&
}
$$
Given $f$, define $i(f)$ in $x\in P$ as $f\cdot c_x$. This map $i(f)$ is monotone since $|P|\cdot G$ is a conical coproduct. The resulting map $i$ is also monotone: $f \leq f'$ implies $f\cdot c_x\leq f'\cdot c_x$ thus $i(f) \leq i(f')$.

Conversely, given $g\colon P\to \ck(G,X)$ in $\Pos$, the morphism $\bar g\colon |P| \cdot G\to X$ given by $\bar g = [g(x)]_{x\in |P|}$ fulfils $\bar g\cdot \bar r_0 \leq \bar g\cdot \bar r_1$ because each pair $x_0 \leq x_1$ in $R$ yields $g(x_0) \leq g(x_1)$ in $\ck(G, X)$.
 Let $j(g) \colon C\to X$ be the unique morphism with
 $$
 \bar g = j(g) \cdot c\,.
 $$
 This defines a monotone map  $j$: if $g\leq h\colon P\to \ck(G,X)$, then $\bar g \leq \bar h$, thus $j(g)\leq j(h)$ by the universal property of $c$.

It is easy to see that $i$ and $j$ are inverse  to each other. And $i$ is natural: given $u\colon X\to X'$, then $i(u\cdot f)$ assigns to $x$ the value $u\cdot f\cdot c_x$, which is what $u\cdot i(f)$ does, too.
 \end{proof} 

Recall that effective objects (in ordinary categories) are those with hom-functor preserving coequalizers of congruences. Here is the enriched variant:

\begin{defi}[{\cite{ARO}}]
An object $G$ of an order-enriched category is \textit{subeffective} if its hom-functor to $\Pos$ preserves coinserters of subcongruences.
\end{defi}
 \begin{example}\label{E:eff} The free algebra $G$ on one generator in a variety $\cv$ of ordered algebras is subeffective. Indeed, its hom-functor is naturally isomorphic to the forgetful functor $U\colon\cv \to \Pos$. It follows from the description of coinserters of subequivalences in the proof of Proposition~\ref{P:subex} that $U$ preserves these coinserters.
 \end{example}
 
 The following proposition has a completely analogous proof to that of Proposition~\ref{P:eff}:
 
 \begin{prop}\label{P:gener} In a category $\ck$ with subkernel  pairs and their coinserters 
 let $G$ be a subregularly projective strong  generator. Then we have the following equivalence:
 \begin{center}
 $G$ subeffective $\Leftrightarrow$ $\ck$ has effective subcongruences.
 \end{center}
 \end{prop}

 \begin{defi}\label{D:subex}
 An order-enriched category is \emph{subexact} if it has 

(1) Subkernel pairs and reflexive coinserters.

(2) Effective subcongruences.
\end{defi}
We have not included stability of subregular epimorphisms under pushouts. The reason is the following lemma

\begin{lem}\label{3.24}
Let $G$ be a subregularly projective strong generator $G$ with copowers in $\ck$. If $\ck$ has subkernel pairs and their  coinserters, then subregular epimorphisms are stable under pullback.
\end{lem}

\begin{proof}
(1) The generator $G$ is proregular: all canonical morphisms $[h] \colon \coprod\limits_{h\colon G\to X} G\to X$ are proregular epimorphisms.  This is proved precisely as Lemma \ref{L:reg}. The only modification is that for the pair $u_0, u_1$ there we assume $u_0\leq u_1$.

(2) $\ck$ has factorization of morphisms as a subregular epimorphism followed by an order-embedding, this is analogous to (1) in the proof of Lemma \ref{L:stable}. The rest is as (2) in that proof.
\end{proof}
Recall subvarietal generators (Def. \ref{de1.4}).

\begin{prop}\label{P:co}
If a category has subkernel pairs, reflexive coinserters and a subvarietal generator, then it is complete and cocomplete.
\end{prop}

\begin{proof}
Let $G$  be a subvarietal generator in a category $\ck$ as bove. We form a small full subcategory $\ca$ of $\ck$ containing $G$ and closed under finite coproducts (which  exist since $\coprod\limits_{n}G$ is the tensor $(n\cdot I) \otimes G$). This subcategory is dense: see \cite[Theorem 3.23]{ARO}. Moreover, $\ck$ has reflexive coequalizers of reflexive pairs $p,q\colon X\to Y$ in $\ca$: the following pair 
$$
\xymatrix@=3pc{
X+X \ar@<.5ex>[r]^{[p,q]}
\ar@<-.5ex>[r]_{[q,p]}
& Y
}
$$
is reflexive, and its coinserter is the coequalizer of $p$ and $q$.

By Corollary \ref{C:refl}, $\ck$ is complete and cocomplete because $\Pos$ is.
\end{proof}

 The following theorem slightly improves and corrects Theorem 4.8 of \cite{ARO}.

 \begin{theorem} An order-enriched category is equivalent to a variety iff it is subexact and has a subvarietal generator.
  \end{theorem}
 
 \begin{proof}
 Every variety satisfies the above conditions by Examples \ref{E:nec} and \ref{E:eff}.
 
 Let $\ck$ be  a category with a generator $G$ as above. We first verify some properties of $\ck$.
 
 (a) $\ck$ has  factorizations  of morphisms as a subregular epimorphism followed by an embedding. Here
  $m\colon A\to B$ is an embedding if $m \cdot u_0\leq m \cdot u_1$ implies $u_0 \leq u_1$ for all $u_0$, $u_1\colon U\to A$. (The proof presented in \cite{ARO} is incomplete.)
 
 Given a morphism $f\colon A\to B$, form the subkernel pair $r_0$, $r_1\colon R\to A$ of $f$. This is a reflexive pair, thus, a coinserter $c\colon A\to C$ exists. For  the unique morphism $m$ with $f=m\cdot c$
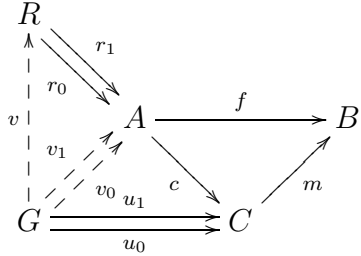
 $$
\xymatrix@C=2pc@R=2pc{
R \ar@<0.5ex>[dr]^{r_1}
\ar@<-0.5ex>[dr]_{r_0} && &\\
& A \ar[rr]^{f} \ar[dr]_{c} && B\\
G \ar@{-->}[uu]^{v} \ar@{-->}@<0.5ex>[ur]^{v_1}
\ar@{-->}@<-0.5ex>[ur]_{v_0}
\ar@<0.5ex>[rr]^{u_1}
\ar@<-0.5ex>[rr]_{u_0}&& C \ar[ur]_{m}&
}
$$
we prove that it is an embedding. Since $G$ is a strong generator, this is equivalent to proving for all $u_0$, $u_1\colon G\to C$ that
$$
m\cdot u_0 \leq m\cdot u_1 \quad \mbox{implies} \quad u_0\leq u_1\,.
$$
Since $G$ is a subregular projective, there exist $v_i\colon G\to A$ with $u_i= c \cdot v_i$. Then $f\cdot v_0 = m\cdot u_0 \leq m\cdot u_1 = f\cdot v_1$
 implies that we have $v\colon G\to R$ with $v_i=r_i\cdot v$. This proves the desired inequality:
 $$
 u_0 = c\cdot r_0 \cdot v\leq c \cdot r_1 \cdot v= u_1\,.
 $$

 (b) The generator $G$ is subregular, i.e., the canonical morphism $[h]$ below is a subregular epimorphism (for every object $X$):
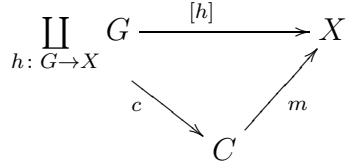
 $$
\xymatrix@C=2pc@R=2pc{
\coprod\limits_{h\colon G\to X} G \ar[rr]^{[h]} \ar[dr]_{c} && X\\
& C \ar[ur]_{m}&
}
$$

Indeed, let $[h]= m\cdot c$ be a factorization as in (a). Since $G$ is a strong generator, and every morphism from it to $X$ factorizes through $[h]$, thus also through  the embedding $m$, it follows that $m$ is invertible.
Thus $[h]$ is a subregular epimorphism.

(c) $\ck$ is complete and cocomplete, see Corollary 3.5, using that $G$ has tensors (Lemma \ref{L:ten}).

(d)  In \cite {ARO} the following signature $\Sigma$ is used: its $n$-ary operations are the morphisms from $G$ to $n\cdot G$:
$$
\Sigma_n = \ck (G, n\cdot G) \qquad (n\in \mathbb N)\,.
$$
We obtain (as proved  in Item(2a) of Thm.~4.8 in loc.~cit.) a full embedding
$$
E\colon \ck \to \Sigma\mbox{-}\Pos
$$
as follows. The algebra $EK$ has the underlying poset $\ck(G,K)$. Given an $n$-ary operation $\sigma\colon G\to n\cdot G$, to every $n$-tuple $f_i \colon G\to K$ the map $\sigma_{EK}$ assigns the following composite
$$
\sigma_{EK} (f_i) \equiv G\xrightarrow{\quad\sigma\quad} n\cdot G \xrightarrow{\quad [f_i]\quad } K\,.
$$
To a morphism $h\colon K\to L$ the functor $E$ assigns the homomorphism
$$
Eh = h\cdot (-) \colon \ck (G,K) \to \ck (G,L)\,.
$$
Let $\bar \ck$ be the closure of $E[\ck]$ under insomorphism in $\Sigma$-$\Pos$. Then $\ck \simeq \bar \ck$, and we use the Birkhoff Variety Theorem to verify that $\bar \ck$ is a variety, thus finishing our proof.

(i) $\bar \ck$ is closed under products because $\ck$ has products by (c), and $E$ clearly preserves limits.

(ii) $\bar \ck$ is closed under subalgebras. The proof presented in \cite{ARO} is incomplete, we present a proof now. A subalgebra of $EK$, for $K\in K$, is a subposet $M\subseteq \ck(G,K)$ closed under the operations. That is, given an $n$-ary symbol $\sigma$, we have
$$
[f_i] \cdot \sigma \in M\quad \mbox{for all}\quad f_0, \dots , f_{n-1} \in M\,.
$$
We are to find an object $L\in\ck$ with $EL\simeq M$.

The canonicalmorphism $[h] \colon \coprod\limits_{h\in M} G \to K$ has  a factorization as a subregular epimorphism $c$ followed by an embedding $m$:
$$
\xymatrix@C=2pc@R=2pc{
\coprod\limits_{h\in M} G \ar[rr]^{[h]} \ar[dr]_{c} && K\\
& C \ar[ur]_{m}&
}
$$
We prove that the ordered algebras $EC$ and $M$ are isomorphic. For that, we verify that in $\Pos$
$$
M = Em [EC]\,.
$$
Since both the subposets $M$ and $Em[EC]$ are  closed under the operations, this implies $M\simeq EC$ in $\Sigma$-$\Pos$, as desired.

The inclusion $M\subseteq Em[EC]$ is obvious: given $h\in M$, the corresponding component $c_h \colon G\to C$  of $c$ above lies in $EC$ and fulfils $h= m\cdot c_h$.

Conversely, we prove
$$
Em(g) = m\cdot g \in M \quad \mbox{for each}\quad g\colon G\to C\,.
$$
Since $G$ is a subregular projective, $g = c \cdot g'$ for some morphism  $g'$:
$$
\xymatrix@C=2pc@R=3pc{
n\cdot G \ar[rr]^{u\cdot G} && M\cdot G \ar[dr]_{c} \ar[rr]^{[h]} && K\\
& G \ar[ul]^{\sigma} \ar[ur]_{g'} \ar[rr]_{g} && C \ar[ur]_{m}&
}
$$
Finite abstractness yields an injection $u\colon n\to G$ (where $n$ denotes the discrete poset $\{ 0, \dots , n-1\}$) such that $g'$ factorizes through $u\cdot G$. We denote by $\sigma$ the factorizing morphism.
Then for $h_i=u(i)\in M$ we get $[h] \cdot (u\cdot G) = [h_i]_{i<n}$ . This  yields 
$$
\sigma_{EK} (h_i)_{i<n} =[h_i]_{i<n} \cdot \sigma\,.
$$
We obtain from the above diagram that
$$
\sigma_{EK} (h_i) = m\cdot g\,.
$$
This concludes the proof of $m\cdot g\in Em[EC]$ since $h_i \in M$ and $M$ is closed under $\sigma_{EK}$.

(iii) $\bar \ck$ is closed under  homomorphic images. See \cite{ARO}, Item (3c) of the proof of Theorem 4.8.
\end{proof}

\begin{corollary}\label{C:main} 
An order-enriched category is equivalent to a variety of ordered algebras iff
it has 

(1) Subkernel pairs and reflexive coinserters.

(2) A subeffective, abstractly finite, strong 
 generator with copowers.
\end{corollary}
This follows from Proposition~\ref{P:gener} and the above theorem.

\end{document}